\documentclass{amsart}
\usepackage{amsmath,amsthm}
\usepackage{amsfonts,amssymb}
\usepackage{accents}
\usepackage{enumerate}
\usepackage{accents,color}
\usepackage{graphicx}
\usepackage{comment}

\newcommand{\lvt}{\left|\kern-1.35pt\left|\kern-1.3pt\left|}
\newcommand{\rvt}{\right|\kern-1.3pt\right|\kern-1.35pt\right|}

\newtheorem{thm}{Theorem}[section]

\newtheorem{lem}[thm]{Lemma}
\newtheorem{prop}[thm]{Proposition}

\theoremstyle{remark}

 \def\la{{\langle}}
 \def\ra{{\rangle}}

 \def\d{\mathrm{d}}
 \def\e{\mathrm{e}}

 \def\sph{{\mathbb{S}^{d-1}}}

 \def\sJ{{\mathsf J}}
 
 \def\sL{{\mathsf L}}
 
 \def\sP{{\mathsf P}}

 \def\sW{{\mathsf W}}

 \def\fB{{\mathfrak B}}
 \def\fD{{\mathfrak D}}

 \def\fR{{\mathfrak R}}

 \def\a{{\alpha}}
 \def\b{{\beta}}
 \def\g{{\gamma}}
 
 \def\t{{\theta}}
 \def\l{{\lambda}}
 
 \def\s{\sigma}
 \def\la{{\langle}}
 \def\ra{{\rangle}}

 \def\bb{{\mathbf b}}

 \def\kb{{\mathbf k}}

 \def\xb{{\mathbf x}}

 \def\Jb{{\mathbf J}}
 
 \def\Lb{{\mathbf L}}
 \def\Pb{{\mathbf P}}
 
 \def\Wb{{\mathbf W}}

 \def\CD{{\mathcal D}}
 
 \def\CH{{\mathcal H}}

 \def\CV{{\mathcal V}}

 \def\BB{{\mathbb B}}

 \def\NN{{\mathbb N}}

 \def\RR{{\mathbb R}}
 \def\SS{{\mathbb S}}
 \def\UU{{\mathbb U}}
 \def\VV{{\mathbb V}}

      \def\proj{\operatorname{proj}}

\def\lla{\langle{\kern-2.5pt}\langle}      
\def\rra{\rangle{\kern-2.5pt}\rangle}

\def\bk{{\boldsymbol{\kappa}}}

\newcommand{\wh}{\widehat}

\def\f{\frac}

\graphicspath{{./}}
\begin{document}
 
\title{Bernstein Inequality on Parabolic Domains}
\author{Yuan~Xu}
\address{Department of Mathematics, University of Oregon, Eugene,
OR 97403--1222, USA}
\email{yuan@uoregon.edu}
\thanks{The author was partially supported by Simons Foundation Grant \#849676}
\date{\today}

\begin{abstract}
Several families of sharp Bernstein inequalities are established on the weighted $L^2$ space over 
parabolic domains, which include bounded or unbounded rotational paraboloids and parabolic surfaces. 
The main tool is a second-order differential operator satisfied by a specific basis of orthogonal polynomials 
in weighted $L^2$ space. 
\end{abstract}

\maketitle

\section{Introduction}
\setcounter{equation}{0}
 
Bernstein inequalities on several classical domains have been revisited in recent studies, due to a 
new revelation of the existence of stronger inequalities than those in the literature. The phenomenon 
was first observed in \cite{X23b} for the triangle, as a byproduct of a study 
on the rotational cone in $\RR^{d+1}$, which motivated a follow-up study \cite{GX} that established sharp 
Bernstein inequalities on the simplex in $\RR^d$ that are stronger than long accepted inequalities (cf. \cite{Dit})
for this classical domain. Even more, stronger inequalities turn out to exist  \cite{BX} even for the unit ball in 
$\RR^d$, one of the most studied domains and a prototype for other rotational domains. 

There are two types of results in these studies, both based on properties of orthogonal polynomials defined 
via a class of weight functions on the given domain. The first type consists of Bernstein inequalities in
weighted $L^p$ norm that hold for all doubling weights, and their proofs are based on highly localized
kernels for a particular weight function. The existence of latter kernels is established via the {\it addition 
formula}, which is the closed-form formula for the reproducing kernel of the orthogonal projection operator, 
mimicking the addition formula for spherical harmonics, that exists only on a few regular domains. The 
second type consists of sharp Bernstein inequalities in the $L^2$ norm, which is sharp in the sense that 
an inequality becomes an equality for certain extremal polynomials. Most of the latter results can be established 
via the {\it spectral operator}, which is a second-order linear derivation operator that has orthogonal 
polynomials as eigenfunctions, with the eigenvalues depending only on the total degree of orthogonal 
polynomials. For the simplex and the unit ball, new sharp Bernstein inequalities arise from new decompositions 
of the classical spectral operators on these domains. 

Both the addition formula and the spectral operator provide powerful tools for analysis, but they exist 
only on a few regular domains. For $d =2$, the existence of the spectral operator was characterized in
\cite{KrSh}. 
For $d > 3$, no classification is known. While the spectral operators for the unit ball and the simplex are classical, 
the one for the rotational cone was discovered only recently \cite{X20}. 

The purpose of the present paper is to establish sharp Bernstein inequalities on rotational parabolic 
domains, which consist of solid paraboloids in $\RR^{d+1}$, defined by  
 $$
    \UU^{d+1} = \left \{(x,t): \|x\| \le \sqrt{t}, \quad 0 \le t \le b, \quad x \in \RR^d\right \},
 $$
and parabolic surfaces $\UU_0^{d+1}$ of these paraboloids, where $b$ could be infinity. Orthogonal
polynomials for $\UU^2$, bounded by a parabola and a line when $b=1$, are semi-classical and 
studied already in \cite{K75}. The orthogonal structure for $d \ge 2$ has been explored more recently 
\cite{X23a}. While an explicit basis of orthogonal polynomials can be given explicitly for two 
classes of weight functions, called the Jacobi and the Laguerre polynomials, they do not 
satisfy a full addition formula, nor do they possess a spectral operator.  The lack of tools explains why 
analysis in this domain is far less developed; see \cite{X23a} for what is known. Although there is no 
spectral operator, the specific orthogonal basis, the Jacobi and the Laguerre polynomials on parabolic 
domains, satisfy a second-order linear differential operator, but with the eigenvalues depending on 
one more index, rather than only on the degree of the polynomials. It turns out that this operator can 
be used to establish sharp Bernstein inequalities in the $L^2$ norm, following more or less the same 
paradigm based on the spectral operator. This leads to the main results of this work. 
 
Our results are restricted only to inequalities in the $L^2$ norm, and they are the first ones established 
on the parabolic domains and new even in the parabolic domain in the plane. At the moment, we are
not aware of an effective tool for establishing inequalities in the $L^p$ norm, for $p \ne 2$, on parabolic 
domains. These $L^2$ inequalities, especially the additional weight functions that accompany the derivatives, 
provide a possible guidance for the formulation of inequalities in the $L^p$ norm.   
 
The paper is organized as follows. In the next section, we discuss how to deduce a sharp Bernstein
inequality in the $L^2$ norm from the spectral operator, and illustrate the approach by reviewing the 
results on the unit ball, which also serve as a preliminary for the latter sections. The sharp Bernstein
inequalities on the paraboloids are established in Section 3, and those on the parabolic surfaces are
established in Section 4.    

\section{Preliminary: sharp $L^2$ Bernstein inequalities}
\setcounter{equation}{0}

Let $\Omega$ be a domain in $\RR^d$ and let $W$ be a non-negative weight function defined on $\Omega$ so that
$$
  \la f, g \ra_{\Omega, W} = \int_\Omega f(x) g (x) W(x) \d x
$$
is a well-defined inner product on $L^2(\Omega, W)$. Let $\Pi_n(\Omega)$ be the space of polynomials, restricted
on $\Omega$, of degree at most $n$ in $d$ variables. For each $n \in \NN_0$, let $\CV_n(\Omega, W)$ denote the 
subspace of $\Pi_n(\Omega)$ that consists of orthogonal polynomials of degree $n$ with respect to the inner product
$\la \cdot,\cdot\ra_{\Omega,W}$. Let $\proj_n: L^2(\Omega,W) \mapsto \CV_n(\Omega, W)$ be the orthogonal 
projection operator. If $\{P_j^n: 1 \le j \le d(n) \}$ is an orthogonal basis of $\CV_n(\Omega, W)$, where 
$d(n):= \dim \CV_n(\Omega, W)$, then the Fourier orthogonal expansion of $f \in L^2(\Omega,W)$ is given by 
$$
  f = \sum_{n=0}^\infty \proj_n f, \quad \quad \hbox{where} \quad 
       \proj_n f = \sum_{j=1}^{d(n)} \frac{\la f, P_j^n\ra_{\Omega, W}}{ \la P_j^n, P_j^n\ra_{\Omega, W}} P_j^n.  
$$

A second-order linear derivation operator $\fD$ is called a spectral operator of $L^2(\Omega, W)$ if there exist 
nonnegative numbers $\l(n)$ such that 
\begin{equation} \label{eq:fD}
 \fD P = - \l(n) P, \qquad \forall P \in \CV_n(\Omega, W), \qquad n =0,1,2, \ldots.
\end{equation}
That is, $\fD$ has eigenvalue $\l(n)$ with eigensapce $\CV_n(\Omega, W)$ for all $n\in \NN_0$. 
The spectral operator exists only for rather special domains $\Omega$ and/or the weight $W$. When it exists, 
it is self-adjoint in $L^2(\Omega, W)$ by orthogonality and \eqref{eq:fD}. Furthermore,
the following lemma holds. 

\begin{lem}\label{lem:main}
Assume the spectral operator $\fD$ exists. For $f\in L^2(\Omega, W)$, 
$$
   - \int_\Omega \fD f(x) \cdot f(x) W(x) \d x = \sum_{m= 0}^\infty \lambda(m) \int_\Omega \left|\proj_m f(x)\right|^2 W(x) \d x. 
$$
In particular, if $f \in \Pi_n(\Omega)$, then 
$$
   \left| \int_\Omega \fD f(x) \cdot f(x) W(x) \d x \right| \le \max_{0 \le m \le n} \lambda(m) \int_\Omega |f(x)|^2 W(x) \d x.
$$
\end{lem}

The identity in the lemma follows from orthogonality since $\proj_m f \in \CV_m(\Omega, W)$ so that 
$\fD \proj_m f  = - \lambda(m)\proj_m f$, whereas the inequality follows from the Parseval identity of
the Fourier orthogonal series in $L^2(\Omega, W)$. The inequality, simple as it looks, is one of the main 
ingredients for deriving the sharp Bernstein inequality in $L^2$-norm. The other ingredient is the explicit 
self-adjoint form of the operator $\fD$. 

We illustrate how the latter works by working with an example: the unit ball $\BB^d$ of $\RR^d$ 
equipped with the classical weight function 
$$
  W_\BB^\mu(x) = (1-\|x\|^2)^{\mu-\f12}, \qquad \mu > -\tfrac12.   
$$
Several orthogonal bases for $\CV_n(\BB^d,W_\BB^\mu)$ can be given explicitly. One of them is 
given in terms of the product of the Gegenbauer polynomials by parametrizing $\BB^d$ 
in the Cartesian coordinates. For $x = (x_1, \ldots, x_d) \in \mathbb{R}^d$, define 
by $\xb_j$ a truncation of $x$, namely $\mathbf{x}_0 = 0$ and $\mathbf{x}_j = (x_1, \ldots, x_j)$, $1 \le j \le d$. 
Associated with $\kb = (k_1,\ldots, k_d)$, define $\kb^j := (k_j, \ldots, k_d)$, $1 \le j \le d$, and $\kb^{d+1} := 0$.
Then for $\kb \in \mathbb{N}_0^d$ with $|\kb| = k_1+\ldots + k_d =n$, define polynomials $\Pb_\kb^n$ by 
\begin{align}\label{secondb}
\Pb_\kb^n(x)  = \prod_{j=1}^d (1 - \| \mathbf{x}_{j-1} \|^2)^{\kb_j / 2}  C_{\kb_j}^{\lambda_j}\!\bigg(
      \frac{x_j}{\sqrt{1 - \| \mathbf{x}_{j-1} \|^2}} \bigg),
\end{align}
where $\lambda_j = \mu + |\kb^{j+1}| + \frac{d - j+1}{2}$. Then $\{\Pb_\kb^n: |\kb| = n\}$ is an orthogonal basis 
of $\mathcal{V}_n(W_\mu, \BB^d)$ \cite[Proposition 5.2.2]{DX}. Another orthogonal basis can be given in terms 
of the Jacobi polynomials, with $2\|x\|^2-1$ as argument, and spherical harmonics \cite[(5.2.4)]{DX} by 
parametrizing $\BB^d$ in spherical polar coordinates. 

The spectral operator $\fD_\BB^\mu$ for $L^2(\BB^d, W_\BB^\mu)$ is given by \cite[(5.2.3)]{DX}
\begin{equation} \label{eq:fD_Bd}
  \fD_\BB^\mu:= \Delta - \la x, \nabla \ra^2 - (2\mu + d-1) \la x, \nabla \ra 
\end{equation}
which satisfies 
\begin{equation} \label{eq:eigenB}
  \fD^\mu_\BB P = - n (n+ 2 \mu + d) P, \qquad \forall P \in \CV_n(\BB^d, W_\BB^\mu). 
\end{equation}
To derive sharp Bernstein inequalities, we need to rewrite $\fD_\BB^\mu$ in different forms. There are two
such forms. The first one is classical (cf. \cite[Proposition 7.1]{DaiX1}) 
\begin{align}\label{fDBd2}
  \fD^{\mu}_\BB =  \frac{1}{W_\BB^\mu(x)}\sum_{i=1}^{d} \partial_{i}\,\!\big(W_\BB^{\mu+1}(x)\,\partial_{i}\big) 
    + \fD_\SS,
\end{align}
where $\fD_\SS$ is the spherical part of the Laplace operator, and its restriction on the unit sphere is the 
Laplace-Beltrami operator $\Delta_0$, which furthermore satisfies \cite[(1.8.3)]{DaiX}
\begin{align}\label{fD_SS}
\fD_\SS = \sum_{1 \le i < j \le d} D_{i,j}^2, \qquad \hbox{where} \quad  D_{i,j}: = x_i \partial_j - x_j \partial_i,
\end{align}
where the operators $D_{i,j}$ are angular derivatives on the sphere \cite[(1.8.1)]{DaiX} since if 
$(x_i,x_j) = r_{i,j}(\cos \t_{i,j}, \sin \t_{i,j})$, then $D_{i,j} = \frac{\partial}{\partial \t_{i,j}}$, and they 
are self-adjoint operators of $L^2(\sph)$. It follows, in particular, that \cite[Section 1.8]{DaiX}
\begin{equation} \label{eq:intSS}
  \int_{\sph} \Delta_0 f(\xi) g (\xi) \d \s_\SS(\xi) = -  \sum_{1\le i< j \le d} \int_{\sph} D_{i,j} f(\xi) D_{i,j} g(\xi) \d \s_\SS (\xi).
\end{equation}

One immediate consequence of \eqref{fDBd2} and \eqref{eq:intSS} is the following integral identity 
derived via integration by parts, 
\begin{align*}
  - \int_{\BB^d} \fD^\mu_\BB f(x) g(x) W_\BB^\mu(x) \d x = & \sum_{i=1}^{d} \int_{\BB^d} (1-\|x\|^2) \partial_{i}f(x) 
   \partial_{i}g(x) W_\BB^\mu (x) \d  \\ & + \sum_{1 \le i< j \le d} \int_{\BB^d} D_{i,j} f(x) D_{i,j} g(x)W_\BB^\mu(x) \d x,        
\end{align*}
which shows that $\fD_\BB^\mu$ is self-adjoint. Moreover, it leads to, setting $g = f$ and using Lemma \ref{lem:main},
a set of Bernstein inequalities in $L^2$ norm (\cite{K2022} and \cite{BX}).  Let $\|\cdot\|_{\mu,2}$ denote the norm
of $L^2(\BB^d, W_\BB^\mu)$. 
 
\begin{thm} \label{thm:B-ineq0Ball}
Let $d \ge 2$, $\mu > -\f12$, $n = 0,1,2,\ldots$ and $f \in \Pi_n^d$. Then
\begin{equation} \label{eq:B-ineq0ball}
 \sum_{i=1}^d \left \| \sqrt{1-\|x\|^2} \partial_i f \right \|_{\mu,2}^2 + 
  \sum_{1\le i<j\le d} \left \|D_{i,j} f \right \|_{\mu,2}^2 \le n(n+2\mu+d) \|f \|_{\mu,2}^2
\end{equation}
and the equality holds if and only if $f \in \CV_n(\BB^d, W_\BB^\mu)$. 
Furthermore, the following two inequalities are also sharp,
\begin{align}
 \sum_{i=1}^d \left \| \sqrt{1-\|x\|^2} \partial_i f \right \|_{\mu,2}^2 \le n(n+2\mu+d) \|f \|_{\mu,2}^2, \quad \text{if $n$ is even}
  \label{eq:B-ineq0aball} \\ 
 \sum_{i=1}^d \left \| \sqrt{1-\|x\|^2} \partial_i f \right \|_{\mu,2}^2 \le (n(n+2\mu+d) -d+1) \|f \|_{\mu,2}^2, \quad \text{if $n$ is odd}.
  \label{eq:B-ineq0aOdd}
\end{align}
\end{thm}

The sharpness of these inequalities is shown by choosing $f$ as a specific polynomial in $\CV_n(\BB^d, W_\BB^\mu)$.
For example, to show \eqref{eq:B-ineq0aball} is sharp, we can choose $f$ as 
$$
   f(x) = P_m^{(\mu,\frac{d-1}{2})} \left(2\|x\|^2-1\right), \qquad \hbox{$n = 2m$},
$$
where $P_m^{(a,b)}$ is the Jacobi polynomials, which is a rotational invariant polynomial of degree $n$ in 
$\CV_n(\BB^d,W_\BB^\mu)$.

The second decomposition of the spectral operator $\fD_\mu^\BB$ appears more recently in \cite{BX}, which states that 
\begin{align} \label{fDBd3}
\fD^\mu_\BB = \frac{1}{W_\mu(x)} \left[ \frac{1}{\|x\|^d} \la x ,\nabla\ra \left(\|x\|^{d-2}(1-\|x\|^2) W_\mu(x)  \la x ,\nabla\ra \right)    \right] + \frac{1}{\|x\|^2} \CD_\SS.
\end{align}
It implies, in particular, another integral identity that shows $\fD_\BB^\mu$ is self-adjoint, 
\begin{align*}
  -\int_{\BB^d} \fD_\BB^\mu f(x) \cdot g(x) W_\BB^\mu(x) \d x 
     \,= &  \int_{\BB^d} \la x ,\nabla \ra f(x) \cdot \la x ,\nabla \ra g(x) (1-\|x\|^2)W_\BB^\mu(x) \frac{\d x}{\|x\|^2} \\
      &  +  \sum_{1\le i<j\le d} \int_{\mathbb{B}^d} D_{i,j} f(x)  D_{i,j} g(x) W_\BB^\mu(x) \frac{\d x}{\|x\|^2}. \notag
\end{align*}
Setting $g = f \in \Pi_n^d$ and using Lemma \ref{lem:main}, the above identity yields another set of sharp 
Bernstein inequalities \cite{BX}. 
 
\begin{thm} \label{thm:BInequality_new}
Let $d \ge 2$, $n = 0,1,2,\ldots$ and $f \in \Pi_n^d$. Then
\begin{align} \label{eq:BB-ineq1}
     \left \|\frac{\sqrt{1-\|x\|^2}}{\|x\|}  \la x ,\nabla \ra f\right \|_{\mu,2}^2
       +  \sum_{1\le i<j\le d} \left \| \frac{1}{\|x\|} D_{i,j} f \right\|_{\mu,2}^2 \le 
            n(n+2\mu+d) \|f\|_{\mu,2}^2,
\end{align}
and the equality holds if and only if $f \in \CV_n(W_\mu, \BB^d)$. Furthermore, the following two inequalities are also sharp,
\begin{align} 
\left \|\frac{\sqrt{1-\|x\|^2}}{\|x\|}  \la x ,\nabla \ra f\right \|_{\mu,2} \le \sqrt{n(n+2\mu+d)}  \|f\|_{\bk,2}, \quad \text{if $n$ is even} \label{eq:BB-ineq1a} \\
\left \|\frac{\sqrt{1-\|x\|^2}}{\|x\|}  \la x ,\nabla \ra f\right \|_{\mu,2} \le \sqrt{n(n+2\mu+d) -d+1}  \|f\|_{\bk,2} \quad \text{if $n$ is odd}. \label{eq:BB-ineq1b}
\end{align}
\end{thm}

As illustrated by these results on the unit ball, our main ingredients are the spectral operators and their self-adjoint forms.
For $d = 1$, both are well-known as they pertain to classical orthogonal polynomials, which are associated with the 
Jacobi weight $(1-t)^\a (1+t)^\b$ on $[-1,1]$, the Laguerre weight $t^\a \e^{-t}$ on $\RR_+$, and the Hermite
weight $\e^{-t^2}$ on $\RR$. For $d =2$, the classification in \cite{KrSh} shows that, up to an affine change of 
variables, there are essentially five classes, three from products of Laguerre and Hermite weights on the product 
domains, one on the unit disk, and one on the triangle. While no classification is known for $d \ge 3$, there are more 
distinct cases for $d \ge 3$. Besides the straightforward extension of the unit ball $\BB^d$ and the standard simplex 
$\triangle^d$ with classical Jacobi weight functions \cite{DX}, spectral operators were recently discovered for the 
rotational conic domains \cite{X21}. They also exist on quadratic surfaces, such as the unit sphere and the conic 
surfaces. Furthermore, the weight functions on these rotationally invariant domains can be extended to the Dunkl 
weight functions that are invariant under the reflection group. 

Besides the unit ball, the sharp $L^2$ Bernstein inequalities for the product Hermite and/or Laguerre weights 
in higher dimensions were established in \cite{KS}, which also fits into the approach outlined above. 
The sharp Bernstein inequalities on the rotational conic domains are established, among other things, 
in \cite{X23b}, which includes new inequalities on the triangle as a special case, and sharp inequalities on 
the simplex are established in \cite{GX}, following the above approach. 

\section{Bernstein inequalities on paraboloids}
\setcounter{equation}{0}

We consider Bernstein inequalities on the paraboloids defined by 
$$
  \UU^{d+1} = \left\{(x,t) \in \RR^{d+1}: \|x\| \le \sqrt{t}, \quad 0 \le t \le b, \, x \in \RR^d \right\},
$$
which is rotationally invariant around the $t$-axis, where $b$ is a positive number, usually 
chosen as $b =1$, or positive infinity. Paramezrising the domain by $x =\sqrt{t} y$, $y\in \BB^d$, it follows readily that 
\begin{equation}\label{intUU=}
  \int_{\UU^{d+1} } f(x,t) \d x \d t  = \int_0^b \int_{\|x\|^2 \le t} f(x,t) \d x \d t =
    \int_0^b t^{\frac{d}{2}} \int_{\BB^d} f\left(\sqrt{t} y, t\right) \d y \d t. 
\end{equation}

For the solid domain, we denote $\Pi^{d+1}:= \Pi(\UU^{d+1})$, which is the usual space of polynomials
of degree at most $n$ in $d+1$ variables. We need to consider the orthogonality of polynomials with 
respect to a weight function on $\UU^{d+1}$, which consists of two cases.  

\subsection{Jacobi polynomials on bounded paraboloid} 
We consider the bounded paraboloid with $b =1$; that is, $0 \le t \le 1$ in the definition of $\UU^{d+1}$.  
For $\g > -1$ and $\mu> -\f12$,  we define a weight function $W_\UU^{\g,\mu}$ on $\UU^{d+1}$,
$$
    \Wb_\UU^{\g,\mu}(x,t) :=  (1-t)^\g (t - \|x\|^2)^{\mu-\f12}, \qquad  (x,t) \in \UU^{d+1} 
$$
and, accordingly, the inner product on the paraboloid defined by 
$$
  \la f,g\ra_\UU^{\g,\mu} = \int_{\UU^{d+1}} f(x,t) g(x,t) \Wb_\UU^{\g,\mu}(x,t) \d x\d t,
$$
For $n =0,1,2,\ldots$, let $\CV_n(\UU^{d+1},\Wb_\UU^{\g,\mu})$ denote the space of orthogonal polynomials 
of degree at most $n$ under this inner product. Then $\dim \CV_n(\VV^{d+1}, \Wb_\UU^{\g,\mu}) = \binom{n+d}{n}$. 
An orthogonal basis of this space is given in \cite{OX, X23a} in terms of the Jacobi polynomials and an 
orthogonal basis on the unit ball. 

Let $\{\Pb_{\kb}^m: |\kb| = m, \, \kb\in \NN_0^d\}$ be an orthogonal basis with parity of $\CV_m(\BB^d,W_\BB^\mu)$,
such as the one given in \eqref{secondb}. For $0 \le m \le n$, define  
\begin{equation} \label{eq:Jbasis_UU}
 \Jb_{m,\kb}^n(x,t) = P_{n-m}^{(m+\mu+\f{d-1}{2},\g)}(1-2 t) t^{\f{m}{2}} 
    \Pb_{\kb}^m \left(\frac{x}{\sqrt{t} }\right), \quad |\kb| = m, \,\, 0\le m \le n.
\end{equation}
Then $\{\Jb_{m,\kb}^n: |\kb | = m, \, 0 \le m  \le n, \, \kb \in \NN_0^d\}$ is an orthogonal basis of 
$\CV_n(\UU^{d+1}, \Wb_\UU^{\g,\mu})$. 
We call $\Jb_{m,\kb}^n$ the Jacobi polynomials on the paraboloid. 
 
There is a second-order differential operator, denoted by $\fD_{\Jb}^{\g,\mu}$, that the Jacobi polynomials 
on the paraboloid satisfy, as shown in \cite{X23a}. 

\begin{prop}
Let $\g > -1$ and $\mu > -\f12$. Then $u= \Jb_{m,\kb}^n$ in \eqref{eq:Jbasis_UU} satisfies the differential 
equation 
\begin{align}\label{eq:diff-eqnUU}
 & \fD_{\Jb}^{\g,\mu} u : = \left[ t (1-t) \partial_{tt} + (1-t) \la x,\nabla_x\ra \partial_t + \frac{1}{4} (1-t)\Delta_x  \right. \\
      & \qquad   \qquad  \quad
      +\left.   \left(\mu+\tfrac{d+1}{2}\right)(1-t)\partial_t   - \frac{\g+1}{2} (2t \partial_t +  \la x,\nabla_x\ra) \right]  u 
       = - \l_{m,n} u \notag
\end{align}
where $\lambda_{m,n}$ is given by, for $0 \le m \le n$, 
$$
 \l_{m,n} = n \big( n+\mu+ \g + \tfrac{d+1}{2}\big) - m \big(n + \mu+\tfrac{\g + d}{2}\big). 
$$
\end{prop}

The operator $\fD_{\Jb}^{\g,\mu}$, however, is not a spectral operator since its eigenvalues depend on both 
$n$ and $m$. In other words, the equation \eqref{eq:diff-eqnUU} depends on this particular basis of the Jacobi 
polynomials; it does not hold for every basis of $\CV_n(\UU^{d+1}, \Wb_{\UU}^{\b,\g})$. Nevertheless, it can
be used to establish Bernstein inequalities on the paraboloid. We start with the following alternative for
Lemma \ref{lem:main}. 

\begin{lem}\label{lem:mainUU}
Let $\g > -1$ and $\mu > -\f12$. Then, for $f \in \Pi_n(\UU^{d+1}) = \Pi^{d+1}$, 
$$
  - \int_{\UU^{d+1}} \fD_{\Jb}^{\g,\mu} f(x,t) \cdot f(x,t) \Wb_{\UU}^{\g,\mu} (x,t) \d x \d t 
       \le  \lambda_{0,n} \int_{\UU^{d+1}} |f(x,t)|^2 \Wb_{\UU}^{\g,\mu} (x,t) \d x \d t.
$$
\end{lem}

\begin{proof}
In terms of the orthogonal basis $\Jb_{m,\kb}^n$, the Fourier expansion of $f$ is 
$$
  f(x,t) =  \sum_{m=0}^n \sum_{j=0}^m \sum_{|\kb| = j} \wh f_{j, \kb}^m \Jb_{j,\kb}^m(x,t), \qquad \wh f_{j,\kb}^m 
     = \frac{\la f, \Jb_{j,\kb}^m\ra_\UU^{\g,\mu}}{\la \Jb_{j,\kb}^m, \Jb_{j,\kb}^m\ra_\UU^{\g,\mu}}.
$$
By \eqref{eq:diff-eqnUU} and the orthogonality, 
\begin{align*}
   - \int_{\UU^{d+1}} \fD_{\Jb}^{\g,\mu} f(x,t) \cdot f(x,t) \Wb_\UU^{\g,\mu}(x,y) \d x \d t
     = \sum_{m=0}^n \sum_{j=0}^m  \l_{j, m} \sum_{|\kb| = j} \left| \wh f_{j, \kb}^m \right|^2 \\
     \le \l_{0,n} \sum_{m=0}^n \sum_{j=0}^m  \sum_{|\kb| = j} \left| \wh f_{j, \kb}^m \right|^2
     = \l_{0,n} \int_{\UU^{d+1}} |f(x,t)|^2 \Wb_{\UU}^{\g,\mu} (x,t) \d x \d t
\end{align*}
by the Parseval identity, where we have used $\l_{j,m} \le \l_{0,m} \le \l_{0,n}$. 
\end{proof}

Using this lemma, we can then establish the Bernstein inequalities on the paraboloid if the operator 
$\fD_{\Jb}^{\g,\mu}$ can be written in an appropriate self-adjoint form. The main step for the latter
is the following theorem. 

\begin{thm} 
Let $\g > -1$ and $\mu > -\f12$. The operator $ \fD_{\Jb}^{\g,\mu}$ can be rewritten as
\begin{align} \label{eq:fD_UU2}
 \fD_{\Jb}^{\g,\mu} =  \fR_\Jb^{\g,\mu}  + \frac{1-t}{4t} \fD_\BB^{\mu, (y)}
\end{align}
where $\fD_\BB^{\mu, (y)}$ denotes the spectral operator $\fD_\BB^\mu$, defined in \eqref{eq:fD_Bd}, acting on 
the variable $y = \frac{x}{\sqrt{t}} \in \BB^d$, and $\fR_\Jb^{\g,\mu} $ is defined by 
\begin{align*}
 \fR_\Jb^{\g,\mu} =  \frac{1}{t^{\f d 2} \Wb_\UU^{\g,\mu}(x,t)} \left( \partial_t + \frac{1}{2t} \la x ,\nabla_x\ra \right) 
  \left[t^{\f{d}{2} +1} \Wb_\UU^{\g+1,\mu}(x,t) \left( \partial_t + \frac{1}{2t} \la x ,\nabla_x\ra \right) \right]. \notag
 \end{align*}
In particular, $\fD_\Jb^{\g,\mu}$ is self-adjoint on $L^2(\UU^{d+1}, \Wb_{\UU}^{\g,\mu})$. Furthermore, 
\begin{align}\label{eq:int_fDU1}
  &  - \int_{\UU^{d+1}} \fR_{\Jb}^{\g,\mu} f(x,t)\cdot g(x,t) \Wb_{\UU}^{\g,\mu}(x,t) \d x \d t \\
 =  \int_{\UU^{d+1}} & t(1-t)  \left( \partial_t + \frac{1}{2t} \la x ,\nabla_x\ra \right) f(x,t)  \left( \partial_t + \frac{1}{2t} \la x ,\nabla_x\ra \right) g(x,t) \Wb_{\UU}^{\g,\mu}(x,t) \d x \d t \notag
\end{align}
and
\begin{align}\label{eq:int_fDU2}
 - & \int_{\UU^{d+1}}  \frac{1-t}{4t} \fD_{\BB}^{\mu, (y)} f(x,t)\cdot g(x,t) \Wb_{\UU}^{\g,\mu}(x,t) \d x \d t \\
  & =  \int_0^1 \frac{1-t}{4t} \left[ \int_{\BB^d} \fD_\BB^{\mu, (y)} f\left(\sqrt{t} y, t\right ) \cdot g\left(\sqrt{t} y, t\right)
   \Wb_\BB^\mu (y) \d y\right]  t^{\f{d-1}{2} + \mu}(1-t)^\g \d t. \notag
\end{align}
\end{thm}

\begin{proof}
The main hurdle lies in recognizing the correct form. The verification comes down to heavy computation of
derivatives, tedious but not difficult, and our proof will be succinct. We start from an observation 
$$
   \left(\partial_t + \frac{1}{2t} \la x,\nabla_x\ra \right) (t-\|x\|^2)^{\mu - \frac12}
    = \frac{\mu - \f12}{t} (t-\|x\|^2)^{\mu-\f12},
$$
which follows from a quick computation. This identity is then used to take the derivatives in $\fR_\Jb^{\g,\mu}$ to
deduce, after simplification, 
\begin{align*}
 \fR_\Jb^{\g,\mu} 
  = \,& \left[ \mu+ \frac {d+1} 2 - \left(\mu+\g+ \frac {d+3} 2 \right)t \right]  \left( \partial_t + \frac{1}{2t} \la x ,\nabla_x\ra \right) \\
    & \,\, + t(1-t)  \left( \partial_t + \frac{1}{2t} \la x ,\nabla_x\ra \right)^2, 
\end{align*}
Furthermore, the second term on the right-hand side satisfies 
\begin{equation*}
   t \left( \partial_t + \frac{1}{2t} \la x ,\nabla_x\ra \right)^2 = t \partial_{tt} + \la x, \nabla_x \ra \partial_t + 
      \frac{1}{2t} \left( \frac 12 \la x, \nabla_x\ra^2 - \la x, \nabla_x\ra \right),
\end{equation*}
so that we can deduce, after rearranging terms and using the definition of $\fD_\UU^{\g,\mu}$ in \eqref{eq:diff-eqnUU}, that 
\begin{align*}
 \fR_\Jb^{\g,\mu}& =  \fD_\UU^{\g,\mu} -\frac{1-t}4 \Delta_x + \left (\mu+\frac{d+1}{2}\right )\frac{1-t}{2t} \la x, \nabla_x \ra
  +   \frac{1-t}{2t} \left( \frac 12 \la x, \nabla_x\ra^2 - \la x, \nabla_x\ra \right)\\
  & =  \fD_\UU^{\g,\mu} - \frac{1-t}{4t} \left(t \Delta_x - \la x, \nabla_x \ra^2 - (2\mu+d-1) \la x, \nabla_x \ra \right).
\end{align*}  
Now, setting $y = \frac{x}{\sqrt{n}}$, it follows that $\frac{\partial}{\partial y_i} = \sqrt{t}  \frac{\partial}{\partial x_i}$
and $\la x, \partial_x\ra = \la y, \partial_y\ra$, so that 
$$
t \Delta_x - \la x, \nabla_x \ra^2 - (2\mu+d-1) \la x, \nabla \ra = \Delta_y - \la y, \nabla_y \ra^2 - (2\mu+d-1) \la y, \nabla_y \ra
 = \fD_\BB^{\mu, (y)}.
$$
Together, these identities prove \eqref{eq:fD_UU2}. 

Let $I_\fR$ be the integral on the left-hand side of \eqref{eq:int_fDU1}. Parameterizing the integral with $x = \sqrt{t} y$
as in \eqref{intUU=}, and using the identity
\begin{equation}\label{eq:dt_UU}
   \frac{\d}{\d t} f\left(\sqrt{t} y, t\right ) = \left(\partial_t + \frac{1}{2t} \la x, \nabla_x \ra \right) f\left(\sqrt{t} y, t\right )
\end{equation}
and $\Wb_\UU^{\g,\mu}(y, 1) = 0$, we deduce via integration by parts, 
\begin{align*}
  I_{\fR}\,& = \int_{\BB^d} \int_0^1 \frac{\d} {\d t} \left[ t^{\frac{d}{2}+1} \Wb_\UU^{\g+1,\mu}\left(\sqrt{t} y, t\right) 
     \frac{\d} {\d t}f\left(\sqrt{t} y, t\right ) \right] g \left(\sqrt{t} y, t\right ) \d t \d y \\
     & = - \int_{\BB^d} \int_0^1\left[ t^{\frac{d}{2}+1} \Wb_\UU^{\g+1,\mu}\left(\sqrt{t} y, t\right) 
   \frac{\d} {\d t}  f\left(\sqrt{t} y, t\right ) \right]   \frac{\d} {\d t} g \left(\sqrt{t} y, t\right ) \d t \d y,
\end{align*} 
which is equal to the right-hand side of \eqref{eq:int_fDU1} upon using \eqref{eq:dt_UU}, and \eqref{intUU=}. 

Finally, the identity \eqref{eq:int_fDU2} is an immediate consequence of rewriting the integral via $x = \sqrt{t} y$
and using $\Wb_\UU^{\g,\mu} (\sqrt{t} y, t) = (1-t)^\g t^{\mu-\f12} W_\BB^\mu (y)$. Since $\fD_\BB^\mu$ is self-adjoint,
the self-adjointness of $\fD_\Jb^{\g,\mu}$ follows from the two identites \eqref{eq:int_fDU1} and \eqref{eq:int_fDU2}. 
\end{proof}

We are now ready to state the Bernstein inequalities for $L^2(\UU^{d+1}, \Wb_\UU^{\g,\mu})$. We denote 
the norm of this space by 
$$
  \| f \|_{\Wb_{\UU}^{\g,\mu}} : = \|f\|_{L^2(\UU^{d+1}, \Wb_\UU^{\g,\mu})}, \qquad \g >  -1, \quad \mu > - \tfrac 12.
$$
While the identity \eqref{eq:int_fDU1} is of the appropriate form, we need the self-adjoint form for $\fD_\BB^\mu$
in \eqref{eq:int_fDU2}, for which we have two choices as shown in Section 2, and state two sets of inequalities
accordingly. First, we use the decomposition of $\fD_\BB^\mu$ in \eqref{fDBd2}. 

\begin{thm} \label{thm:B-ineqUJ}
Let $d \ge 1$, $n = 0,1,2,\ldots$ and $f \in \Pi_n^{d+1}$. Then
\begin{align} \label{eq:U-ineqJ1}
& \left\| \sqrt{ t(1-t)}  \left( \partial_t + \frac{1}{2t} \la x ,\nabla_x\ra \right) f\right \|_{\Wb_{\UU}^{\g,\mu}}^2 
  +  \sum_{i=1}^d \left \| \frac{\sqrt{1-t} \sqrt{t-\|x\|^2}}{2 \sqrt{t}} \partial_{x_i} f \right \|_{\Wb_{\UU}^{\g,\mu}}^2 \\
 & \qquad\qquad \quad 
  +  \sum_{1\le i<j\le d} \left \| \frac{\sqrt{1-t}}{2 \sqrt{t}}  D_{i, j}^{(x)} f \right \|_{\Wb_{\UU}^{\g,\mu}}^2
       \le n\left(n+\g+\mu+\frac{d+1}{2}\right) \|f \|_{\Wb_{\UU}^{\g,\mu}}^2 \notag
\end{align} 
and the equality holds if and only if $f= \Jb_{0,\kb}^n$ in \eqref{eq:Jbasis_UU}. 
Furthermore, the following inequality is also sharp, 
\begin{align}\label{eq:U-ineqJ2}
\left\| \sqrt{ t(1-t)}  \left( \partial_t + \frac{1}{2t} \la x ,\nabla_x\ra \right) f\right \|_{\Wb_{\UU}^{\g,\mu}}
  \le \sqrt{n\left(n+\g+\mu+\frac{d+1}{2}\right)} \, \|f \|_{\Wb_{\UU}^{\g,\mu}}.
\end{align}
\end{thm}
 
\begin{proof}
Let $I_{\fB}$ denote the integral in the left-hand side of \eqref{eq:int_fDU2}. Using the decomposition of \eqref{fDBd2} 
of $\CD_\BB^\mu$, we obtain
\begin{align*}
 I_{\fB} & =  \int_0^1 \frac{1-t}{4t} \left[ \int_{\BB^d} \sum_{i=1}^d \partial_{y_i} f \left(\sqrt{t} y, t\right )
    \partial_{y_i} g\left(\sqrt{t} y, t\right)
   W_\BB^{\mu+1} (y) \d y\right]  t^{\f{d-1}{2} + \mu}(1-t)^\g \d t \\
  & +   \int_0^1 \frac{1-t}{4t} \left[ \int_{\BB^d} \sum_{1\le i< j \le d} D_{i,j}^{(y)} f \left(\sqrt{t} y, t\right )  
    D_{i,j}^{(y)} g\left(\sqrt{t} y, t\right)  W_\BB^{\mu} (y) \d y\right]  t^{\f{d-1}{2} + \mu}(1-t)^\g \d t. 
\end{align*} 
Since $\frac{\partial}{\partial y_i} = \sqrt{t} \frac{\partial}{\partial x_i}$ for $y = \frac{x}{\sqrt{t}}$, it follows
readily that $D_{i,j}^{(y)} =  D_{i,j}^{(x)}$, so that 
\begin{align*}
 I_{\fB} & =  \int_{\UU^{d+1}} \frac{1-t}{4t}   \sum_{i=1}^d \partial_{x_i} f(x,t) \partial_{x_i} g(x,t) 
 \Wb_\UU^{\g,\mu+1} (x,t) \d x \d t \\ 
  & +   \int_{\UU^{d+1}} \frac{1-t}{4t} \sum_{1\le i< j \le d} D_{i,j}^{(x)} f (x,t) D_{i,j}^{(x)} g(x,t) \Wb_\UU^{\g,\mu}(x,t) \d x \d t
\end{align*} 
Seting $g = f$ in this identiy and in the identity \eqref{eq:int_fDU1}, it follows from \eqref{eq:fD_UU2} that
\begin{align*}
 - \int_{\UU^{d+1}} & \fD_\Jb^{\g,\mu} f(x,t) \cdot f(x,t) \Wb_\UU^{\g,\mu} (x,t) \d x \d t  \\
  & =  \int_{\UU^{d+1}} t(1-t)  \left| \left( \partial_t + \frac{1}{2t} \la x ,\nabla_x\ra \right) f(x,t) \right|^2
   \Wb_{\UU}^{\g,\mu}(x,t) \d x \d t \\
  & +  \int_{\UU^{d+1}} \frac{1-t}{4t}   \sum_{i=1}^d \left| \partial_{x_i} f(x,t) \right|^2 
 \Wb_\UU^{\g,\mu+1} (x,t) \d x \d t \\ 
  & +   \int_{\UU^{d+1}} \frac{1-t}{4t} \sum_{1\le i< j \le d} \left| D_{i,j}^{(x)} f (x,t) \right |\Wb_\UU^{\g,\mu}(x,t) \d x \d t,
\end{align*}
from whcih the inequality \eqref{eq:U-ineqJ1} follows immediately from Lemma \ref{lem:mainUU} and it has 
\eqref{eq:U-ineqJ2} as a corollary. Both these inequalities are sharp, as can be seen by
choosing $f(x,t)= J_{0,\kb}^n(x,t) = P_n^{(\mu+ \f{d-1}{2}}(1-2 t)$. 
\end{proof}
 
It is worth mentioning that \eqref{eq:U-ineqJ1} also yields the inequality 
\begin{align}\label{eq:U-ineqJ3B}
 \sum_{i=1}^d \left \| \frac{\sqrt{1-t} \sqrt{t-\|x\|^2}}{2 \sqrt{t}} \partial_{x_i} f \right \|_{\Wb_{\UU}^{\g,\mu}}
 \le \sqrt{n\left(n+\g+\mu+\tfrac{d+1}{2}\right)} \left \|f \right \|_{\Wb_{\UU}^{\g,\mu}}.   
\end{align}
Although the extra weight functions in front of the derivative looks to be appropriate, we do not know if this inequality is sharp. 
For its analog on the unit ball, the extremal function for even $n$ is the rotationally invariant orthogonal polynomial 
$p_n(\|x\|)$ in $\CV(\BB^d, W_\BB^\mu)$. The corresponding polynomial on the paraboloid is 
$ t^{n/2} p_n\left(\frac{\|x\|}{t}\right)$, which corresponds to, however, an orthogonal polynomial of the form $\Jb_{n, \kb}^n$, for which 
$$
 - \fD_{\Jb}^{\g,\mu} \Jb_{n, \kb}^n = \l_{n,n} = \frac{\g+1}{2} n \, \Jb_{n, \kb}^n,
$$ 
where the constant in the right-hand is of order $n$ instead of $n^2$, so that it cannot be used to show that 
the inequality \eqref{eq:U-ineqJ3B} is sharp even in terms of the power of $n$.  \medskip

Next, we use the second decomposition \eqref{fDBd3} of the spectral operator $\fD_\BB^\mu$. 

\begin{thm} \label{thm:B-ineqUJ2}
Let $d \ge 1$, $n = 0,1,2,\ldots$ and $f \in \Pi_n^{d+1}$. Then
\begin{align} \label{eq:U-ineqB}
& \left\| \sqrt{ t(1-t)}  \left( \partial_t + \frac{1}{2t} \la x ,\nabla_x\ra \right) f\right \|_{\Wb_{\UU}^{\g,\mu}}^2 +
  \left\| \frac{\sqrt{1-t} \sqrt{t-\|x\|^2} }{2 \sqrt{t} \|x\|} \la x ,\nabla_x\ra f \right \|_{\Wb_{\UU}^{\g,\mu}}^2  \\
 & \qquad\qquad  
  +  \sum_{1\le i<j\le d} \left \| \frac{\sqrt{1-t}}{2 \|x\|}  D_{i, j}^{(x)} f \right \|_{\Wb_{\UU}^{\g,\mu}}^2
       \le n\left(n+\g+\mu+\frac{d+1}{2}\right) \|f \|_{\Wb_{\UU}^{\g,\mu}}^2 \notag
\end{align} 
and the equality holds if and only if $f= \Jb_{0,\kb}^n$ in \eqref{eq:Jbasis_UU}. 
\end{thm}

\begin{proof}
Again, let $I_{\fB}$ denote the integral in the left-hand side of \eqref{eq:int_fDU2}. Using the integration by parts
formula of $\CD_\BB^\mu$, stated below \eqref{fDBd3}, we obtain
\begin{align*}
 I_{\fB} & =  \int_0^1 \frac{1-t}{4t} \left[ \int_{\BB^d}
   \la y, \nabla_y \ra f \left(\sqrt{t} y, t\right ) \cdot \la y, \nabla_y \ra g\left(\sqrt{t} y, t\right) 
    W_\BB^{\mu+1} (y) \frac{\d y}{\|y\|^2} \right]  t^{\f{d-1}{2} + \mu}(1-t)^\g \d t \\
  & +   \int_0^1 \frac{1-t}{4t} \left[ \int_{\BB^d} \sum_{1\le i< j \le d} D_{i,j}^{(y)} f \left(\sqrt{t} y, t\right )  
    D_{i,j}^{(y)} g\left(\sqrt{t} y, t\right)  W_\BB^{\mu} (y) \frac{ \d y}{\|y\|^2} \right]  t^{\f{d-1}{2} + \mu}(1-t)^\g \d t. 
\end{align*} 
Since $\la y, \nabla_y\ra = \la x, \nabla_x\ra$ and $D_{i,j}^{(y)} =  D_{i,j}^{(x)}$ for $y = \frac{x}{\sqrt{t}}$,
we can write the integals as over $\UU^{d+1}$ to obgtain
\begin{align*}
 I_{\fB} & =  \int_{\UU^{d+1}} \frac{1-t}{4t} \la x, \nabla_x \ra f(x,t) \cdot \la x, \nabla_x \ra g(x, t) 
    \Wb_\UU^{\mu+1,\g} (x,t) \frac{\d x}{\|x\|^2} \d t \\
  & +   \int_{\UU^{d+1}} \frac{1-t}{4\|x\|^2} \sum_{1\le i< j \le d} D_{i,j}^{(x)} f (x,t) D_{i,j}^{(x)} g(x,t) \Wb_\UU^{\g,\mu}(x,t) \d x\d t,
\end{align*} 
from whcih the inequality \eqref{eq:U-ineqB} follows immediately as in the proof of \eqref{eq:U-ineqJ1}. 
\end{proof}

We note that the two inequalities, \eqref{eq:U-ineqJ1} and \eqref{eq:U-ineqB}, are comparable, but neither is 
stronger. Indeed, their first terms are equal,
the second term on the left-hand side of \eqref{eq:U-ineqJ1} dominates, by the Cauchy-Schwartz inequality, 
the second term on the left-hand side of \eqref{eq:U-ineqB}, yet the third term on the lleft-hand side 
of \eqref{eq:U-ineqJ1} is dominated, by $\|x\| \le \sqrt{t}$ on $\UU^{d+1}$, by the third term on the left-hand 
side of \eqref{eq:U-ineqB}.

\subsection{Laguree polynomials on unbounded paraboloid} 
We consider the unbounded paraboloid, with $b=\infty$, or $0 \le t < \infty$, in the definition of $\UU^{d+1}$.  
For $\mu> -\f12$,  we define a weight function $W_\UU^{\mu}$ on unbounded $\UU^{d+1}$ by
$$
    \Wb_\UU^{\mu}(x,t) :=  e^{-t} (t - \|x\|^2)^{\mu-\f12}, \qquad  (x,t) \in \UU^{d+1} 
$$
and, accordingly, the inner product on the paraboloid defined by 
$$
  \la f,g\ra_\UU^{\mu} = \bb_{\mu} \int_{\UU^{d+1}} f(x,t) g(x,t) \Wb_\UU^{\mu}(x,t) \d x\d t. 
$$
For $n=0,1,2,\ldots$, let $\CV_n(\UU^{d+1}, \Wb_\UU^\mu)$ be the space of orthogonal polynomials
of degree at most $n$. Then $\dim \CV_n(\VV^{d+1}, \Wb_\UU^{,\g,\mu}) = \binom{n+d}{n}$. 
As in the case of the Jacobi polynomials on the finite paraboloid, an orthogonal basis of this space 
can be given in terms of the Laguerre polynomials and an orthogonal basis on the unit ball. 

Let again $\{\Pb_{\kb}^m: |\kb| = m, \, \kb\in \NN_0^d\}$ be an orthogonal basis with parity of 
$\CV_m(\BB^d,W_\BB^\mu)$, for example, the basis given in \eqref{secondb}. For $0 \le m \le n$, define  
\begin{equation} \label{eq:Lbasis_UU}
 \Lb_{m,\kb}^n(x,t) = L_{n-m}^{m+\mu+\f{d-1}{2}}(1-2 t) t^{\f{m}{2}} 
    \Pb_{\kb}^m \left(\frac{x}{\sqrt{t} }\right), \quad |\kb| = m, \,\, 0\le m \le n,
\end{equation}
where $L_n^\a$ is the Laguerre polynomial of degree $n$ that is orthogonal with respect to the weight
function $t^\a e^{-t}$ on $\RR_+$ for $\a > -1$.  
Then $\{\Lb_{m,\kb}^n: |\kb | = m, \, 0 \le m  \le n, \, \kb \in \NN_0^d\}$ is an orthogonal basis of 
$\CV_n(\UU^{d+1}, \Wb_\UU^{\mu})$. 

Like the case of the Jacobi polynomials on the paraboloid, there is no spectral operator in $L^2(\UU^{d+1}, \Wb_\UU^\mu)$,
but the orthogonal basis of $\Lb_{m,\kb}^n$ satsify an equation defined by a second-order differential operator. 

\begin{prop}
Let $\mu > -\f12$. Then $\Lb_{m,\kb}^n$ in \eqref{eq:Lbasis_UU} satisfies the differential equation 
\begin{align}\label{eq:diff-eqnUU_L}
 & \fD_{\Lb}^{\mu} \, u: = \left[ t \partial_{tt} + \la x,\nabla_x\ra \partial_t + \frac{1}{4}\Delta_x  
     -  \f12 \la x \nabla_x \ra  \right]  u 
       = - \left (n-\frac m 2 \right) u 
\end{align}
for $|\kb| = m$ and $0 \le m \le n$. 
\end{prop}

\begin{proof}
Let $\a = \mu + \frac{d-1}{2}$. To simplify the notation, we write $u(x,t) = g(t) H(x,t)$, where 
$g(t) = L_{n-m}^{m+\a}(t)$ and $H(x,t) = t^{\f m 2}\Pb_{\kb}^m(\frac{x}{\sqrt{t}})$. As observed in \cite[(4.8) and (4.9)]{X23a},
$H$ satisifes 
\begin{equation} \label{eq:diff_H}
   2 t \frac{\partial H}{\partial t} + \la x, \nabla_x \ra H = m H \quad \hbox{and} \quad
 2 t \frac{\partial^2 H}{\partial t^2} + \la x, \nabla_x \ra \frac{\partial H}{\partial t}  = (m-2) H. 
\end{equation}
Using the first of these two identities and taking derivatives of $u$, we obtain 
\begin{align*}
  t \partial_{tt} u + \la x, \nabla_x \ra u \,& = t g'' H + \left(2t \frac{\partial H}{\partial t} + \la x ,\nabla_x \ra H\right) g' 
    +  g  \left( t \frac{\partial^2 H}{\partial t^2}+\la x, \nabla_x \ra \frac{\partial H}{\partial t}\right) \\
      & =  t g'' H +m g' H 
    +  g  \left( t \frac{\partial^2 H}{\partial t^2}+\la x, \nabla_x \ra \frac{\partial H}{\partial t}\right).
 \end{align*}
The Laguerre polynomial $L_n^\a$ satisfies $t y'' + (\a+1-t) y' = - n y$, so that $g = L_{n-m}^{m+\a}$ satisifies
the equation 
$$
      t g''(t) + (m+\a +1-t)g'(t) = - (n-m) g(t).  
$$
Consequently, using $\partial_t u= g' H + g \frac{\partial H}{\partial t}$, it follows readily that 
\begin{align*}
   t \partial_{tt} u &\, + \la x, \nabla_x \ra u + (\a + 1 -t) \partial_t u  \\
  & = -(n-m) u + g  \left( t \frac{\partial^2 H}{\partial t^2}+\la x, \nabla_x \ra \frac{\partial H}{\partial t}
     + (\a+1-t)\frac{\partial H}{\partial t} \right) \\
  & = -(n-m) u - \frac14 \Delta_x u - t g \frac{\partial H}{\partial t},     
\end{align*}
where the second identity follows from the second ideitity in \eqref{eq:diff_H} and an equation on $H$ 
deduced from the spectral equation \eqref{eq:eigenB}, as shown in the proof of \cite[Prop. 4.2]{X23a}. 
Finally, using the first idendity in \eqref{eq:diff_H}, we deduce
$$
 t g \frac{\partial H}{\partial t} = \f12 g (m H - \la x,\nabla_x H) = \frac{m}2 u - \frac 12 \la x, \nabla_x \ra g.
$$
Combining the last two identities proves \eqref{eq:diff-eqnUU_L}. 
\end{proof}

The operator $\fD_\UU^\mu$ for the Laguerre polynomials on the paraboloid is not a spectral operator
since its eigenvalues depend on both $n$ and $m$, just like the operators $\fD_\UU^{\g,\mu}$ in 
\eqref{eq:diff-eqnUU} for the Jacobi polynomials on the paraboloid. For this operator, the lemma below
is an analog of Lemma \ref{lem:mainUU} with a verbatim proof. 

\begin{lem}\label{lem:mainUU2}
Let $\mu > -\f12$. Then, for $f \in \Pi^{d+1}$, 
$$
  - \int_{\UU^{d+1}} \fD_{\Lb}^{\mu} f(x,t) \cdot f(x,t) \Wb_{\UU}^{\mu} (x,t) \d x \d t 
       \le  n  \int_{\UU^{d+1}} |f(x,t)|^2 \Wb_{\UU}^{\mu} (x,t) \d x \d t.
$$
\end{lem}

We use this lemma to establish the Bernstein inequalities on the unbounded paraboloid. As in the Jacobi 
case, we need to rewrite $\fD_{\Lb}^{\mu}$ in an appropriate self-adjoint form stated below. 

\begin{thm} 
Let $\mu > -\f12$. The operator $ \fD_{\Lb}^{\mu}$ can be rewritten as
\begin{align} \label{eq:fD_UU_B}
 \fD_{\Lb}^{\mu} =  \fR_\Lb^{\mu}  + \frac{1}{4t} \fD_\BB^{\mu, (y)},
\end{align}
where $\fD_\BB^{\mu, (y)}$ denotes the spectral operator $\fD_\BB^\mu$, defined in \eqref{eq:fD_Bd}, acting on 
the variable $y = \frac{x}{\sqrt{t}} \in \BB^d$, and $\fR_\Lb^{\mu} $ is defined by 
\begin{align*}
 \fR_\Lb^{\mu} =  \frac{1}{t^{\f d 2} \Wb_{\mu}(x,t)} \left( \partial_t + \frac{1}{2t} \la x ,\nabla_x\ra \right) 
  \left[t^{\f{d}{2} +1} \Wb_{\mu}(x,t) \left( \partial_t + \frac{1}{2t} \la x ,\nabla_x\ra \right) \right]. \notag
\end{align*}
In particular, $\fD_\Lb^{\mu}$ is self-adjoint on $L^2(\UU^{d+1}, \Wb_{\UU}^{\mu})$. Furthermore, 
\begin{align}\label{eq:int_fDU1B}
  &  - \int_{\UU^{d+1}} \fR_{\Lb}^{\mu} f(x,t)\cdot g(x,t) \Wb_{\UU}^{\mu}(x,t) \d x \d t \\
 &  =   \int_{\UU^{d+1}}  t \left( \partial_t + \frac{1}{2t} \la x ,\nabla_x\ra \right) f(x,t)  
    \left( \partial_t + \frac{1}{2t} \la x ,\nabla_x\ra \right) g(x,t) \Wb_{\UU}^{\mu}(x,t) \d x \d t \notag
\end{align}
and an analog of \eqref{eq:int_fDU2} holds with $1-t$ replaced by 1 in the nominators and $(1-t)^\g$ 
replaced by $\e^{-t}$. 
\end{thm}

\begin{proof}
To prove \eqref{eq:fD_UU_B}, we take derivatives in $\fR_\Lb^\mu$ and follow the steps as in the proof 
of \eqref{eq:fD_UU2}, where the main ingredients have already been taken care of. Also, the proof
of \eqref{eq:int_fDU1B} is similar to that of \eqref{eq:int_fDU1}. We omit the details. 
\end{proof}

We are now ready to state the Bernstein inequalities on the unbounded paraboloid. Denote the norm
of $L^2(\UU^{d+1}, \Wb_\UU^\mu)$ by 
$$
   \|f\|_{\Wb_\UU^\mu} =  \|f\|_{L^2(\UU^{d+1}, \Wb_\UU^\mu)}. 
$$
Using the decomposition of $\fD_\BB^\mu$ in \eqref{fDBd2} gives an analog of Theorem \ref{thm:B-ineqUJ}.

\begin{thm} \label{thm:B-ineqUL}
Let $\mu > -\f12$, $d \ge 1$, $n = 0,1,2,\ldots$ and $f \in \Pi_n^{d+1}$. Then
\begin{align} \label{eq:U-ineqL1}
 \left\| \sqrt{ t} \left( \partial_t + \frac{1}{2t} \la x ,\nabla_x\ra \right) f\right \|_{\Wb_{\UU}^{\mu}}^2 
   & +  \sum_{i=1}^d \left \| \f{\sqrt{t-\|x\|^2}}{2 \sqrt{t}} \partial_{x_i} f \right \|_{\Wb_{\UU}^{\mu}}^2 \\
  & +  \sum_{1\le i<j\le d} \left \| \frac{1}{2 \sqrt{t}}  D_{i, j}^{(x)} f \right \|_{\Wb_{\UU}^{\mu}}^2
       \le n  \|f \|_{\Wb_{\UU}^{\mu}}^2 \notag
\end{align} 
and the equality holds if and only if $f= \Lb_{0,\kb}^n$ in \eqref{eq:Jbasis_UU}. 
Furthermore, the following inequality is also sharp, 
\begin{align}\label{eq:U-ineqL2}
\left\| \sqrt{ t}\left( \partial_t + \frac{1}{2t} \la x ,\nabla_x\ra \right) f\right \|_{\Wb_{\UU}^{\mu}} \le \sqrt{n} \, \|f \|_{\Wb_{\UU}^{\mu}}.
\end{align}
\end{thm}

And, using the decomposition of $\fD_\BB^\mu$ in \eqref{fDBd3} gives an analog of Theorem \ref{thm:B-ineqUJ2}. 

\begin{thm} \label{thm:B-ineqUL2}
Let $\mu> -\f12$, $d \ge 1$, $n = 0,1,2,\ldots$ and $f \in \Pi_n^{d+1}$. Then
\begin{align} \label{eq:U-ineqLB}
 \left\| \sqrt{t}  \left( \partial_t + \frac{1}{2t} \la x ,\nabla_x\ra \right) f\right \|_{\Wb_{\UU}^{\mu}}^2 & +
  \left\| \frac{ \sqrt{t-\|x\|^2} }{2 \sqrt{t} \|x\|} \la x ,\nabla_x\ra f \right \|_{\Wb_{\UU}^{\mu}}^2  \\
 & 
  +  \sum_{1\le i<j\le d} \left \| \frac{1}{2 \|x\|}  D_{i, j}^{(x)} f \right \|_{\Wb_{\UU}^{\mu}}^2
       \le n \|f \|_{\Wb_{\UU}^{\mu}}^2 \notag
\end{align} 
and the equality holds if and only if $f= \Lb_{0,\kb}^n$ in \eqref{eq:Jbasis_UU}. 
\end{thm}

In both cases, the proof follows exactly as in the case of the Jacobi weight functions. 
\section{Bernstein inequalities on parabolic surfaces}
\setcounter{equation}{0}

In this section, we consider Bernstein inequalities on the parabolic surface 
$$
\UU_0^{d+1}: =\left \{(x,t): \|x\|^2 =  t, \,\, 0 \le t \le b, \,\, x \in \RR^d \right \},
$$  
which is the surface of the paraboloid $\UU^{d+1}$, where $b =1$ or $b= +\infty$. 
Let $\d \s$ be the Lebesgue measure on $\UU_0^{d+1}$. Paramezrising the surface by 
$x =\sqrt{t} \xi$ with $\xi \in \sph$, it follows readily that 
$$
  \int_{\UU_0^{d+1} } f(x,t) \d \s(x,t)  = 
    \int_0^\infty t^{\frac{d-1}{2}} \int_{\sph} f\left(\sqrt{t}\xi, t\right) \d\s_{\SS}(\xi)\d t, 
$$
where $\d_\SS$ is the Lebesgue measure on the unit sphere $\sph$. Like in the case of the paraboloid, 
we consider the bounded and unbounded parabolic surfaces separately. 

\subsection{Jacobi polynomials on bounded parabolic surface} 
Let $\UU_0^{d+1}$ be the bounded parabolic surface with $b = 1$. For $\g  > -1$, we define the Jacobie 
weight function 
$$
 \sW_{\UU_0}^\g (t) := t^{-\f12} (1-t)^\g,  \qquad 0 \le t \le 1,
$$
and the inner product on the bounded parabolic surface defined accordingly by
$$
  \la f,g\ra_{\UU_0}^\g = \int_{\UU_0^{d+1}} f(x,t) g(x,t) \sW_{\UU_0}^\g(t) \d\s(x,t).
$$
Let $\Pi_n(\UU_0^{d+1})$ be the space of polynomials of degree at most $n$ in $d+1$ variables restricted 
on the surface $\UU_0^{d+1}$, determined by replacing all occurrences of $\|x\|^2$ with $t$. For $n =0,1,2,\ldots$, 
let $\CV_n(\UU_0^{d+1}, \sW_{\UU_0}^\g)$ be the space of orthogonal polynomials of degree $n$ with respect to 
the inner product $ \la \cdot, \cdot\ra_{\b,\g}$ on the parabolic surface. Then its dimension is the same as that of
the space of spherical harmonics on $\SS^d$, 
$$
  \dim \CV_n\left(\UU_0^{d+1}, \sW_{\UU_0}^\g\right) = \binom{n+d}{n} - \binom{n+d-2}{n-2}
$$
and $\Pi_n(\UU_0^{d+1})$ is an orthogonal direct sum of $\CV_m\left(\UU_0^{d+1}, \sW_{\UU_0}^\g\right)$ for 
$0 \le m \le n$. 

An orthogonal basis of $\CV_n(\UU_0^{d+1}, \sW_{\b,\g})$ is explicitly given in terms of the Jacobi polynomials 
and spherical harmonics \cite{OX} with norms given in \cite[Prop. 3.1]{X23a}. Recall that spherical harmonics 
are restrictions of homogeneous harmonic polynomials on the unit sphere, and they are orthogonal on the sphere
with respect to the surface measure. Let $\CH_n^d$ denote the space of spherical harmonics of degree at most $n$ in 
$d$-variables. 

\begin{prop}
Let $\g > -1$. Let $\{Y_\ell^m: 1 \le \ell \le \dim \CH_m^d\}$ be an orthogonal basis of $\CH_m^d$. 
For $0 \le m \le n$, define  
\begin{equation} \label{eq:polyJ_UU0}
   \sJ_{m,\ell}^n(x,t) = P_{n-m}^{(m +\frac{d-1}{2},\g)}(1-2 t) t^{\f{m}{2}}Y_\ell^m\left(\f{x}{\sqrt{t}}\right). 
\end{equation}
Then $\{\sJ_{m,\ell}^n: 0\le m \le n, \, 1 \le \ell \le \dim \CH_m^d\}$ is an orthogonal basis of 
$\CV_n(\UU_0^{d+1}, \sW_{\UU_0}^\g)$.  
\end{prop}

We call polynomials in $\sJ_{m,\ell}^n$ in \eqref{eq:polyJ_UU0} the Jacobi polynomials on the parabolic surface. 
Setting $x = \sqrt{t} \xi$ with $\xi \in \sph$ and using that $Y_\ell^m$ is homogeneous, we can write
$$
  \sJ_{m,\ell}^n (x,t) = f_{m,n}(t) Y_\ell^m(\xi) \quad \hbox{with}\quad f_{m,n}(t) = P_{n-m}^{(m +\frac{d-1}{2},\g)}(1-2 t) t^{\f{m}{2}}.
$$
Using this expression, it was shown in \cite[Prop. 3.2]{X23a} that the Jacobi polynomials on the parabolic surface
satisfy a differential equation.
 
\begin{prop}\label{prp:pde}
Let $\g > -1$. Then $\sJ_{m,\ell}^n$ in \eqref{eq:polyJ_UU0} satisfies the differential equation
\begin{align}\label{eq:J_U0_diff}
 \fD_\sJ^{\g}  :=  \left[   t (1-t) \frac{\d^2}{\d t^2} +\left(\tfrac{d}{2} - (\g + \tfrac{d}2 + 1)t \right) \frac{\d}{\d t} + \frac{1-t}{4 t} \Delta_0^{(\xi)}
    \right] u = - \l_{m,n} u,
\end{align}
where $\Delta_0^{(\xi)}$ is the Laplace-Beltrami operator acting on $\xi = x/\sqrt{t} \in \sph$ and
$$
  \l_{m,n}  =  n \big( n+\g + \tfrac{d}{2}\big) - m \big(n + \tfrac{\g + d-1}{2}\big).   
$$
\end{prop}

We note that the derivative $\frac{\d} {\d t} f$ satisfies, by chain rule, 
\begin{equation} \label{eq:diff_t =}
\frac{\d} {\d t} f(x,t) = \frac{\d} {\d t}  f\left(\sqrt{t} \xi,t\right)  = \left(\frac{1}{2 t} \la x, \nabla_x \ra + \partial_t\right) f(x,t).  
\end{equation} 

In contrast to orthogonal polynomials on the unit ball and the conic surfaces \cite{DX, X21}, the eigenvalues
$\l_{m,n}$ in \eqref{eq:J_U0_diff} depend on both $m$ and $n$, so that $\CV_n(\UU_0^{d+1}, \sW_{\b,\g})$
is not an eigenspace of the differential operator $\fD^{\g}_{\sP}$, in contrast to the unit sphere and the conic
surfaces. Nevertheless, the operator is self-adjoint in $L^2\left(\UU_0^{d+1}, \sW_{\b,\g}\right)$. 
We use this operator to establish the Bernstein inequality based on the following lemma. 

\begin{lem}\label{lem:mainUU0}
Let $\g > -1$. Then, for $f \in \Pi_n(\UU_0^{d+1})$,  
$$
  - \int_{\UU_0^{d+1}} \fD_{\sJ}^{\g} f(x,t) \cdot f(x,t) \sW_{\UU_0}^\g(x,t) \d \s(x, t) 
       \le  \l_{0,n} \int_{\UU)_0^{d+1}} |f(x,t)|^2\sW_{\UU_0}^\g(x,t) \d \s(x, t).
$$
\end{lem}

The proof of this lemma uses the Fourier orthogonal expansion of $f \in L^2(\UU_0^{d+1}, \sW_{\UU_0}^\g)$, 
and follows the same argument as the proof of Lemma \ref{lem:mainUU}. 

For establishing the Bernstein inequalities on the parabolic surface, we now need to rewrite $\fD_\sJ^\mu$ in a
a self-adjoint form. 

\begin{thm} \label{thm:fD-UU0}
For $\g > -1$, the differential oeprator $\fD_{\sJ}^\g$ satisifes
\begin{align}\label{eq:J_U0_diff2}
 \fD_{\sJ}^\g = \fR_{\sJ}^\g + \frac{1-t}{4 t} \Delta_0^{(\xi)}
\end{align}
where the operator $\fR_\sJ^\g$ is defined by
$$
 \fR_{\sJ}^\g = \frac1{t^{\f {d-1} 2}  \sW_{\UU_0}^\g(t) }\frac{\d}{\d t} \left(t^{\f {d+1}{2}}  \sW_{\UU_0}^{\g+1}(t) \frac{\d}{\d t} \right).
$$
In particular, for $f, g  \in L^2\left(\UU_0^{d+1},  \sW_{\UU_0}^\g\right)$, 
\begin{align} \label{eq:int_U01}
  - \int_{\UU_0^{d+1}} &  \fD_{\sJ}^{\mu} f(x,t)  g (x,t)   \sW_{\UU_0}^\g (t) \d \s (x,t) \\
    & =  \int_{\UU_0^{d+1}} t(1-t) \frac{\d}{\d t} f(x,t) \frac{\d}{\d t} g (x,t)  \sW_{\UU_0}^\g (t) \d \s (x,t) \notag \\ 
    & +  \sum_{1 \le i < j \le d} \int_{\UU_0^{d+1}} \frac{1-t}{4t} D_{i,j}^{(\xi)} f(x,t) D_{i,j}^{(\xi)} g (x,t) \sW_{\UU_0}^\g (t) \d \s (x,t). 
    \notag
\end{align}
\end{thm}

\begin{proof}
The verification of \eqref{eq:J_U0_diff2} follows from a straightforward computation of taking the derivative of
$\fR_{\sJ}^\g$ and comparing with \eqref{eq:J_U0_diff}. The integral identity follows from integration by parts, since  
\begin{align*}
&  \int_{\UU_0^{d+1}}  (\fR_{\sJ}^\mu f)(x,t) g(x,t) \sW_{\UU_0}^\g (t) \d \s (x,t) \\
&   = \int_0^1 \int_{\sph} \frac{\d}{\d t} \left(t^{\f {d+1}{2}} \sW_{-\f12, \g+1}(t) \frac{\d}{\d t} \right)f\left(\sqrt{t} \xi,t\right)\cdot
    g\left(\sqrt{t} \xi,t\right) \d \s_\SS(\xi) \d t \\
& = - \int_0^1 \int_{\sph}  \left(t^{\f {d+1}{2}} \sW_{-\f12, \g+1}(t) \frac{\d}{\d t} \right)f\left(\sqrt{t} \xi,t\right)\cdot
    \frac{\d}{\d t} g\left(\sqrt{t} \xi,t\right) \d \s_\SS(\xi) \d t \\
& = - \int_{\UU_0^{d+1}} t(1-t)  \frac{\d}{\d t}f(x,t)\frac{\d}{\d t}g(x,t)  \sW_{\UU_0}^\g (t) \d \s (x,t),
\end{align*}
and the integral for $\Delta_0^{(\xi)}$ follows from \eqref{eq:intSS}. 
\end{proof}

We are now ready to state the sharp Bernstein inequalities on the parabolic surface. Let us denote the norm of
$L^2(\UU_0^{d+1}, \sW_{\UU_0}^\g)$ by 
$$
   \|f \|_{ \sW_{\UU_0}^\g} = \|f\|_{L^2(\UU_0^{d+1},  \sW_{\UU_0}^\g)}
$$

\begin{thm} \label{thm:B-ineqU0J}
Let $\g > -1$, $d \ge 1$, $n = 0,1,2,\ldots$ and $f \in \Pi_n^{d+1}$. Then
\begin{align} \label{eq:U0-ineq1}
 \left\| \sqrt{t(1-t)} \frac{\d}{\d t} f  \right \|_{\sW_{\UU_0}^\g}^2 
    +  \sum_{1\le i<j\le d} \left \| \frac{\sqrt{1-t}}{2 \sqrt{t}} D_{i, j}^{(x)} f \right \|_{ \sW_{\UU_0}^\g}^2  \le 
    n \big( n+\g + \tfrac{d}{2}\big) \|f \|_{ \sW_{\UU_0}^\g}^2 
\end{align} 
and the equality holds if and only if $f = \sJ_{0,\ell}^n$ in \eqref{eq:polyJ_UU0}. 
Furthermore, the following inequality is also sharp, 
\begin{align} \label{eq:U0-ineq2}
 \left\| \sqrt{t(1-t)} \frac{\d}{\d t} f  \right \|_{\sW_{\UU_0}^\g} \le \sqrt{n \big( n+\g + \tfrac{d}{2}\big)}  \|f \|_{ \sW_{\UU_0}^\g}.
\end{align}  
\end{thm}

The proof is straightforward and follows from the intergal identity \eqref{eq:int_U01} with $g = f$ and
Lemma \ref{lem:mainUU0}. We note that the derivative $\frac{\d}{\d t}$ can be replaced by the expression
in the right-hand side of \eqref{eq:diff_t =}, which means, for example, that \eqref{eq:U0-ineq2} can be stated as
\begin{align} \label{eq:U0-ineq3}
\left\| \sqrt{ t (1-t)}\left( \frac{1}{2t} \la x ,\nabla_x\ra +\partial_t \right) f\right \|_{\sW_{\UU_0}^{\g}} 
   \le \sqrt{n \big( n+\g + \tfrac{d}{2}\big)} \, \|f \|_{\sW_{\UU_0}^{\g}}.
\end{align}

\subsection{Laguerre polynomials on unbounded parabolic surface} 
Let $\UU_0^{d+1}$ be the unbounded parabolic surface with $b = +\infty$.We define the Laguerre 
weight function 
$$
    \sW_{\UU_0} (t) := t^{-\f12} \e^{-t},  \qquad t \in \RR_+, 
$$
and the inner product on the unbounded parabolic surface defined accordingly by
$$
  \la f,g\ra_{\UU_0} = \int_{\UU_0^{d+1}} f(x,t) g(x,t) \sW_{\UU_0}(t) \d\s(x,t).
$$
For $n =0,1,2,\ldots$, let $\CV_n(\UU_0^{d+1}, \sW_{\UU_0})$ be the space of orthogonal polynomials of degree
$n$ with respect to the inner product $ \la \cdot, \cdot\ra_{\UU_0}$ on the parabolic surface, which has the same
dimension as its counterpart for the finite parabolic surface. An orthogonal basis of this space can be given
in terms of the Laguerre polynomials and spherical harmonics \cite{OX}, as can be easily verified. 

\begin{prop}
Let $\{Y_\ell^m: 1 \le \ell \le \dim \CH_m^d\}$ be an orthonormal basis of $\CH_m^d$. For $0 \le m \le n$, define  
\begin{equation} \label{eq:polyL_UU0}
   \sL_{m,\ell}^n(x,t) = L_{n-m}^{m +\frac{d-2}{2}}(t) t^{\f{m}{2}}Y_\ell^m\left(\f{x}{\sqrt{t}}\right). 
\end{equation}
Then $\{\sL_{m,\ell}^n: 0\le m \le n, \, 1 \le \ell \le \dim \CH_m^d\}$ is an orthogonal basis of 
$\CV_n(\UU_0^{d+1}, \sW_{\UU_0})$.  
\end{prop}

We call polynomials in $\sL_{m,\ell}^n$ in \eqref{eq:polyL_UU0} the Laguerre polynomials on the parabolic surface. 
These polynomials satisfy a differential equation. 

\begin{prop}\label{prp:pdeL}
The polynomials $\sL_{m,\ell}^n$ in \eqref{eq:polyL_UU0} satisfies the differential equation
\begin{align}\label{eq:L_U0_diff}
 \fD_\sL :=  \left[   t \frac{\d^2}{\d t^2} + \left(\tfrac{d}{2} - t \right) \frac{\d}{\d t} + \frac{1}{4 t} \Delta_0^{(\xi)}
    \right] u = - \left(n - \frac m2\right) u,
\end{align}
where $\xi = \frac{x}{\sqrt{t}} \in \sph$. 
\end{prop}

\begin{proof}
Since $Y_\ell^m$ is homogeneous, setting $x = \sqrt{t} \xi$ with $\xi \in \sph$ leads to 
$$
  \sL_{m,\ell}^n (x,t) = f_{m,n}(t) Y_\ell^m(\xi) \quad \hbox{with}\quad f_{m,n}(t) = L_{n-m}^{m +\a}(t) t^{\f{m}{2}},
$$
where $\a = \frac{d-2}{2}$. Since the Laguerre polynomial $L_n^\b$ satisfies $t y'' + (\b+1-t)y' = -n y$, it 
follows that $L_{n-m}^{m+\a}$ satisifies the equation $t y'' + (m+ \a +1 -t) y' = -(n-m) y$. Using this equation
and taking derivatives, it is straightforward to verify that 
\begin{align*}
 t f_{m,n}''(t)  + (\a+1 -t) f_{m,n}' - \frac{m (m+ 2 \a)}{4 t} f_{m,n} = - \left(n - \frac{m}{2}\right) f_{m,n}. 
\end{align*}
Multiplying the above equation by $Y_\ell^m(\xi)$, we obtain \eqref{eq:J_U0_diff} from $2 \a = d-2$ and the
spectral equaiton $- m(m+d-2) Y_\ell^m = \Delta_0 Y_\ell^m$ for spherical harmoics. 
\end{proof}
 
As in the previous subsection, the derivative $\frac{\d}{\d t}$ is defined as in \eqref{eq:diff_t =}. An analog of
Lemma \ref{lem:mainUU0} holds for $\fD_{\sL}$, which can be used for the Bernstein inequalities on the 
unbounded parabolic surface. We need the self-adjoint form of $\fD_{\sL}$. 

\begin{thm} 
The differential oeprator $\fD_{\sL}$ satisifes
\begin{align}\label{eq:L_U0_diff2}
 \fD_{\sL} = \fR_{\sL} + \frac{1}{4 t} \Delta_0^{(\xi)}
\end{align}
where the operator $\fR_\sL$ is defined by
$$
 \fR_{\sL}^\g = \frac1{t^{\f {d-1} 2} \sW_{\UU_0}(t) } \frac{\d}{\d t} \left(t^{\f {d+1}{2}}  \sW_{\UU_0}(t) \frac{\d}{\d t} \right).
$$
In particular, for $f, g  \in L^2\left(\UU_0^{d+1},  \sW_{\UU_0}^\g\right)$, 
\begin{align} \label{eq:int_U01B}
  - \int_{\UU_0^{d+1}} &  \fD_{\sL}^{\mu} f(x,t)  g (x,t)   \sW_{\UU_0}(t) \d \s (x,t) \\
    & =  \int_{\UU_0^{d+1}} t \frac{\d}{\d t} f(x,t) \frac{\d}{\d t} g (x,t)  \sW_{\UU_0}(t) \d \s (x,t) \notag \\ 
    & +  \sum_{1 \le i < j \le d} \int_{\UU_0^{d+1}} \frac{1}{4t} D_{i,j}^{(\xi)} f(x,t) D_{i,j}^{(\xi)} g (x,t) \sW_{\UU_0}(t) \d \s (x,t). 
    \notag
\end{align}
\end{thm}

This is proved from a straightforward calculation as that for Theorem \ref{thm:fD-UU0}. We omit the details as well as the
proof of the Bernstein inequalities below. Denote the norm of $L^2(\UU_0^{d+1}, \sW_{\UU_0})$ by 
$$
   \|f \|_{ \sW_{\UU_0}} = \|f\|_{L^2(\UU_0^{d+1},  \sW_{\UU_0})}.
$$

\begin{thm} \label{thm:B-ineqU0L}
Let $d \ge 1$, $n = 0,1,2,\ldots$ and $f \in \Pi_n^{d+1}$. Then
\begin{align} \label{eq:U0-ineqB}
 \left\| \sqrt{t} \frac{\d}{\d t} f  \right \|_{\sW_{\UU_0}} ^2 
    +  \sum_{1\le i<j\le d} \left \| \frac{1}{2 \sqrt{t}} D_{i, j}^{(x)} f \right \|_{ \sW_{\UU_0}}^2  
        \le n  \|f \|_{ \sW_{\UU_0}}^2 
\end{align} 
and the equality holds if and only if $f = \sJ_{0,\ell}^n$ in \eqref{eq:polyJ_UU0}. 
Furthermore, the following inequality is also sharp, 
\begin{align} \label{eq:U0-ineqB2}
 \left\| \sqrt{t} \frac{\d}{\d t} f  \right \|_{\sW_{\UU_0}} 
   =\left\| \sqrt{t}\left( \frac{1}{2t} \la x ,\nabla_x\ra +\partial_t \right) f\right \|_{\sW_{\UU_0}}  \le \sqrt{n}  \| f \|_{\sW_{\UU_0}}.
\end{align}  
\end{thm}

The equation in \eqref{eq:U0-ineqB2} follows from \eqref{eq:diff_t =} as in the bounded parabolic surface.


\end{document}